\documentclass[12pt,twoside]{amsart}
\usepackage{amsmath, amsthm, amscd, amsfonts, amssymb, graphicx}
\usepackage{enumerate}
\usepackage[colorlinks=true,
linkcolor=blue,
urlcolor=cyan,
citecolor=red]{hyperref}
\usepackage{mathrsfs}
\newtheorem{theorem}{Theorem}[section]
\newtheorem{lemma}{Lemma}[section]
\newtheorem{remark}{Remark}[section]

\newtheorem{corollary}{Corollary}[section]

\newtheorem{proposition}{Proposition}[section]
\numberwithin{equation}{section}

\begin{document}
	
\title{On a reverse of the Tan-Xie inequality for sector matrices and its applications}
\author{ Leila Nasiri$^{1*}$ and   Shigeru Furuichi$^{2}$}

\begin{abstract}
In  this  short paper,  we  establish a  reverse  of  the  derived  inequalities  for  sector  matrices  by  Tan  and  Xie, with Kantorovich constant.
Then,  as  application  of  our  main theorem,  some  inequalities  for  determinant  and  unitarily  invariant  norm  
are  presented.
\end{abstract}
\maketitle
\pagestyle{myheadings}
\markboth{\centerline {On a reverse of the Tan-Xie inequality}}
{\centerline {L. Nasiri   and   S.  Furuichi }}
\bigskip
\bigskip
\section{Introduction}
Let  $ \mathbb{M}_{n}$ and  $ \mathbb{M}^{+}_{n}$  denote  the  set  of  all  $n \times n$  matrices    and  the  set  of  all  $n \times n$ positive  semidefinite  matrices  with  entries  in  $\mathbb C,$  respectivey.    For  
$  A \in \mathbb{M}_{n},$   the  cartesian  decomposition  of $A$  is  presented  as  
$$A=\Re A+i \Im A,$$
where  $\Re A=\frac{A+A^{*}}{2}$ and  $\Im A=\frac{A-A^{*}}{2i}$  are  the  real  and  imaginary  parts
of $A,$  respectively.  The  matrix   $  A \in \mathbb{M}_{n}$  is  called  accretive,  if 
$\Re A$  is  positive  definite.  
Also,  The  matrix   $  A \in \mathbb{M}_{n}$  is  called  accretive-disipative,  if 
both  $\Re A$  and  $\Im A$  are  positive  definite.    
For  $\alpha \in \left[0,\frac{\pi}{2}\right)$,   define  a  sector  as  follows:
$$S_{\alpha}=\{z\in \mathbb C: \Re z >0, |\Im z| \leq (\Re z) \tan \alpha\}.$$
Here,  we  recall  that  the  numerical  range  of  $A \in \mathbb{M}_{n}$ is defined  by 
$$W(A)=\{x^{*}Ax: x\in {\mathbb C}^{n}, x^{*}x=1\}.$$
 The  matrix   $  A \in \mathbb{M}_{n}$  is  called  sector,  if  whose numerical  range  is  contained  in  
sector  $S_{\alpha}.$  In other words,  
$W(A) \subset S_{\alpha}.$   Clearly,  any  sector  matrice  is  accretive  with  extra  information  about  
the  angle  $\alpha$. 
Since  $W(A) \subset S_{\alpha}$ implies that    
$W(X^{*}A X) \subset S_{\alpha}$
for any nonsingular  matrix $X\in \mathbb{M}_{n}$, also  $W(A^{-1}) \subset S_{\alpha}$,   that  is,  inverse  of  every  
sector  matrice  is  sector.  Indeed, by defintion, $W(A)\subset S_{\alpha}$ is equivalent to $\pm \Im A \leq \tan \alpha \Re A$. The inequality is in the Loewner partial order. 
Therefore,  $\pm X\Im A X^* \leq \left(\tan \alpha\right) X\Re A X^*$ which is equivalent to $W(X^{*}A X) \subset S_{\alpha}$. 
In addition, if we take $X=A^{-1}$, then we have
$$
\pm A^{-1}\frac{A-A^*}{2i}\left(A^{-1}\right)^* \leq \left(\tan\alpha\right) A^{-1}\frac{A+A^*}{2}\left(A^{-1}\right)^*.
$$
Thus we have
$$
\mp\frac{A^{-1}-\left(A^{-1}\right)^*}{2i}\leq \left(\tan\alpha\right) \frac{\left(A^{-1}\right)^*+A^{-1}}{2}
$$
which means $\pm \Im A^{-1}\leq \left(\tan\alpha\right) \Re A^{-1}$.
This is equivalent to $W(A^{-1})\subset S_{\alpha}$.

For  $A,B \in \mathbb{M}^{+}_{n},$  the weighted  geometric  mean,   the weighted  arithmetic  mean and  
the weighted  harmonic  mean  are  defined,  respectively,  as  follows:
$$A \sharp_{v} B=A^{\frac{1}{2}}(A^{-\frac{1}{2}} BA^{-\frac{1}{2}})^{v}A^{\frac{1}{2}},
A \nabla_{v} B=(1-v) A+ v B ,
A! _{v} B=\left((1-v) A^{-1}+ v B^{-1} \right)^{-1}.  
$$
It  is  clear  that  the  following  inequality  holds  the  between  of  the weighted HM-GM-AM: 
\begin{align}\label{f1I}
A! _{v} B \leq A \sharp_{v} B \leq  A \nabla_{v} B.
\end{align}
In  \cite{WLioaWu},  the  authors  obtain   a  reverse  of   the  second  inequality  in  \eqref{f1I} using  the  Kantorovich  constant  
for  every positive  unital  linear  map  
$\Phi$  as  follows:
\begin{align}
\Phi^{2}( A \nabla_{v} B)  \leq  K^{2}(h) \Phi^{2}(A \sharp_{v} B).
\end{align}
For  $\Phi=id,$  it is  obvious  that  
  \begin{align}\label{f2I}
A \nabla_{v} B  \leq  K(h) (A \sharp_{v} B).
\end{align}
The  authors  \cite{RaissouliMMoslehianMSFuruichiS}  defined  the weighted  geometric  mean  for   two  accretive matrices $A,B \in \mathbb{M}_{n}$ 
and  $v \in [0,1]$
as  follows:
$$ A \sharp_{v} B=\frac{\sin v \pi }{\pi} \int^{\infty}_{0} s^{v-1} (A^{-1}+sB^{-1})^{-1}ds.$$
Tan  and  Xie  \cite{FupingTanAntaiXie}  studied  the  inequality  \eqref{f1I} for sector  matrices   $A,B \in \mathbb{M}_{n},$   $v \in [0,1]$ 
and   $\alpha \in \left[0,\frac{\pi}{2}\right)$ 
  and  obtained  the  following result:
\begin{equation}\label{f11}
\cos^{2}(\alpha) \Re (A!_{v}B)\leq \Re (A \sharp_{v} B) \leq  
\sec^{2}(\alpha)
\Re (A\nabla_{v}B).
\end{equation}

Inspired by the nice results  \eqref{f11},
 we  are  going  to  present  a  reverse  of  the  double  inequality  
\eqref{f11}   for  two  sector  matrices $A,B \in   \mathbb{M}_{n}$  and  $ v \in [0,1]$ in  this  short paper.  Moreover,  we establish
some  new  determinant  and  norm  inequalities  using  the  deduced  inequality.
\section{A  reverse  of  the  double  inequality  \eqref{f11}}
Our aim  of  this  section  is  to  establish  a reverse  of  the  double  inequality  \eqref{f11}
which  both  generalize  and  extend  the  obtained  results  in recent  years. 
 To  do  this  work,  we   use    Kantorovich  constant  $K(h) :=\dfrac{(h+1)^2}{4h} \geq 1$ for $h:=\dfrac{M}{m} \geq 1$ with $0<m\leq M$ throughout the paper and  several  lemmas  which  we  list  them  as  follows: 

\begin{lemma}{\bf (\cite{MLin1})}  \label{147}
Let  $A \in \mathbb{M}_{n}$  be  accretive,  then 
\begin{equation}\label{f3}
\Re (A^{-1}) \leq \Re^{-1} (A).
\end{equation}
\end{lemma}
The  next  lemma  is  a  reverse  of  \eqref{f3}.
\begin{lemma}{\bf (\cite{MLin2})} \label{148}
Let  $A \in \mathbb{M}_{n}$ with  $W(A) \subset S_{\alpha}.$ Then  the  following  inequality holds:
\begin{equation}\label{f4}
 \Re^{-1} (A) \leq \sec^{2}(\alpha) \Re (A^{-1}).
\end{equation}
\end{lemma}

\begin{lemma}{\bf (\cite{RBhatiaFKittaneh})}
Let  $A,B \in \mathbb{B}(H)$  be  positive. Then 
\begin{equation}\label{f5}
\| AB \| \leq \frac{1}{4}\|A+B\|^{2}. 
\end{equation}
\end{lemma}

\begin{lemma}{\bf (Choi inequality \cite[p.41]{Bhatia2007})}
Let  $A \in \mathbb{B}(H)$  be  positive and let $\Phi$ be a positive unital linear map. Then  we  have
\begin{equation}\label{choi_ineq}
\Phi^{-1}(A) \leq \Phi(A^{-1}).
\end{equation}
\end{lemma}

\begin{lemma}{\bf (\cite{MBakherad})} \label{lemma_Bakherad}
Let  $A,B \in \mathbb{B}(H)$  be  positive and let $r$ be a positive number. 
Then  $A \leq rB$ is equivalent to $\|A^{1/2}B^{-1/2}\|\leq r^{1/2}$.
\end{lemma}

\begin{theorem}\label{11t}
Let  $A,B \in  \mathbb{M}_{n}$  be  sector, that is,   $W(A),W(B) \subset S_{\alpha}$ for  some  $\alpha \in \left[0,\frac{\pi}{2}\right)$
and $0 \leq  v \leq 1$. Then for every positive unital linear map $\Phi$, we hve the following.
\begin{itemize}
\item[(i)]
If  $0<m I_{n} \leq \Re (A^{-1}), \Re (B^{-1}) \leq MI_{n}.$
 Then,  
\begin{equation}\label{f112}
  \Phi^2\left(\Re (A  \sharp_{v} B)\right)   \leq  
 \sec^{8}(\alpha) K^2(h)   \Phi^2\left(\Re(A!_{v}B\right).
\end{equation}
\item[(ii)]
If  $0<m I_{n} \leq \Re (A), \Re (B) \leq MI_{n}.$
 Then,  
\begin{equation}\label{f212}
 K^{-2}(h)\cos^{8}(\alpha)  
\Phi^2\left( \Re(A\nabla_{v}B)\right)  \leq 
  \Phi^2\left(\Re (A  \sharp_{v} B)\right).
\end{equation}
\end{itemize}
\end{theorem}
\begin{proof}
\begin{itemize}
\item[(i)]
From $0<m I_{n} \leq \Re (A^{-1}), \Re (B^{-1}) \leq MI_{n},$  we  get 
$$ \Re(A^{-1})+Mm \Re(A^{-1})^{-1} \leq  M+m.$$
$$ \Re(B^{-1})+Mm \Re(B^{-1})^{-1} \leq  M+m.$$
If  we  multiply  both  sides  of  the  first  inequality  and  the  second  inequality,  respectively,  by  $1-v$
and  $v$, we  obtain  
$$(1-v) \Re(A^{-1})+(1-v) Mm \Re(A^{-1})^{-1} \leq (1-v)(M+m).$$
$$v \Re(B^{-1})+ v Mm \Re(B^{-1})^{-1} \leq v(M+m).$$ 
As  the  inverse  of  every  sector  matrice  is  sector  again  and  every  
 sector  matrice is  accretive  as  explained  in  Introduction,  
  it  follows  that

\begin{eqnarray}
&& Mm\Re ((1-v) A+v B)+ \Re((1-v) A^{-1}+v B^{-1})  \nonumber  \\
&& \leq  Mm( (1-v)  \Re^{-1} (A^{-1})+v  \Re^{-1} (B^{-1}))+ \Re((1-v) A^{-1}+v B^{-1}) \quad \text{(by \ref{f3})} \nonumber  \\
&& \leq M+m.\label{f26}
\end{eqnarray}
Thus we have,  
\begin{eqnarray}
&&\| \Phi\left(\Re (A  \sharp_{v} B)\right)
 Mm\Phi^{-1}\left( \Re(A!_{v}B)\right)\|  \nonumber \\
 &&  \leq 
\frac{1}{4}\|  Mm \Phi\left(\Re (A  \sharp_{v} B)\right)+\Phi^{-1}\left( \Re(A!_{v}B)\right)\|^{2} \quad \text{(by \eqref{f5})} \nonumber \\
 &&  \leq 
\frac{1}{4}\|  Mm \Phi\left(\Re (A  \sharp_{v} B)\right)+\Phi\left( \Re^{-1}(A!_{v}B)\right)\|^{2} \quad \text{(by \eqref{choi_ineq})} \nonumber \\
 &&  \leq \frac{1}{4}
\|  Mm \Phi\left(\Re (A  \sharp_{v} B)\right)+\sec^{2}(\alpha)
 \Phi\left(\Re((1-v) A^{-1}+ v B^{-1})\right)\|^{2} \quad \text{(by \eqref{f4})}  \nonumber \\
  &&  \leq 
\frac{1}{4} \|  \sec^{2}(\alpha)Mm \Phi\left(\Re ((1-v) A+ v B)\right)+\sec^{2}(\alpha)
 \Phi\left(\Re((1-v) A^{-1}+ v B^{-1})\right)\|^{2} \quad  \text{(by  \eqref{f11})}  \nonumber \\
  && =\frac{1}{4} \sec^{4}(\alpha) \| \Phi\left(Mm \Re ((1-v) A+ v B)+
 \Re((1-v) A^{-1}+ v B^{-1})\right)\|^{2} \nonumber \\
  &&  \leq  \frac{\sec^{4}(\alpha)}{4}(M+m)^{2} \quad \text{(by  \eqref{f26} )}. \nonumber
\end{eqnarray}
\item[(ii)]
In  similar way,  we   have  
\begin{equation}\label{eq2.6_ii}
Mm\left((1-v) \Re^{-1} (A) +v \Re^{-1} (B)\right)+(1-v) \Re (A)+v \Re (B) \leq M+m 
\end{equation}
from the conditions on $\Re (A)$ and $\Re (B)$ in (ii). Thus we have
\begin{eqnarray}
&& \| \sec^{4}(\alpha)\Phi^{-1}\left(\Re (A  \sharp_{v} B)\right)
 Mm \Phi\left(\Re(A\nabla_{v}B)\right)\|  \nonumber \\
&&  \leq \frac{1}{4}
\|  Mm \Phi^{-1}\left(\Re (A  \sharp_{v} B)\right)+\sec^{4}(\alpha)\Phi\left(
\Re(A\nabla_{v}B)\right)\|^{2}  \quad
 \text{(by \eqref{f5})}\nonumber \\
 &&  \leq \frac{1}{4}
\|  Mm \Phi\left(\Re^{-1} (A  \sharp_{v} B)\right)+\sec^{4}(\alpha)\Phi\left(
\Re(A\nabla_{v}B)\right)\|^{2}  \quad
 \text{(by \eqref{choi_ineq})}\nonumber \\
 &&  \leq  \frac{1}{4}
\| \sec^{2}(\alpha) Mm \Phi\left(\Re \left((A  \sharp_{v} B)^{-1}\right)\right)+\sec^{4}(\alpha)\Phi\left(
\Re(A\nabla_{v}B)\right)\|^{2}  \quad  \text{(by   \eqref{f4})}
\nonumber \\
 &&  = \frac{1}{4}
\| \sec^{2}(\alpha) Mm \Phi\left(\Re (A^{-1}  \sharp_{v} B^{-1})\right)+\sec^{4}(\alpha)\Phi\left(
\Re(A\nabla_{v}B)\right)\|^{2}  \nonumber \\ 
&&  \leq  
\frac{1}{4}
\| \sec^{4}(\alpha) Mm \Phi\left(\Re((1-v) A^{-1}+ v B^{-1})\right)+\sec^{4}(\alpha)\Phi\left(
\Re(A\nabla_{v}B)\right)\|^{2} \quad
\text{(by  \eqref{f11})} \nonumber \\ 
&&  \leq  
\frac{1}{4}
\| \sec^{4}(\alpha) Mm\Phi\left( ((1-v) \Re^{-1}( A)+ v \Re^{-1} (B))\right)+\sec^{4}(\alpha)\Phi\left(
\Re((1-v) A+v B)\right)\|^{2} \quad 
 \text{(by  \eqref{f3})}  \nonumber \\ 
   &&  \leq 
\frac{\sec^{8}(\alpha)}{4}(M+m)^{2}.  \quad \text{(by \eqref{eq2.6_ii})} \nonumber
\end{eqnarray}
\end{itemize}
Thus we have the desired results (i) and (ii) by Lemma \ref{lemma_Bakherad}.
\end{proof}

\begin{remark}
The inequalities given in Theorem \ref{11t} give reverses for the inequalities \eqref{f11} when $\Phi$ is an identity map.
In addition, our inequality \eqref{f212} recovers the  inequality  \eqref{f2I} for $\alpha = 0$ and  $\Phi$ is an identity map. 
\end{remark}

\begin{remark}
For  $v=\frac{1}{2},$  
the  inequalities    \eqref{f112} and  \eqref{f212}  recover  
\cite[Theorem2.18]{YangCLuF} and \cite[Theorem2.10]{YangCLuF} ,  respectively. 
This  shows  that  our  results  contain  the wide  class  of  inequalities. 
\end{remark}

\section{Applications  }

Making use  of  the  inequalities \eqref{f112}  and \eqref{f212}, we  prove  some  determinant  inequalities.  For  proving  the  results  of  this  section,  we   need   to  state  the  following  useful  lemmas  
which  the  first lemma is known as the Ostrowski-Taussky  inequality  and  the  second  lemma  is  a its reverse.
\begin{lemma}{\bf (\cite{HornRAJohnsonCR})}  \label{566}
Let  $A \in \mathbb{M}_{n}$  be  accretive. Then    
\begin{equation}\label{f7}
\det (\Re A) \leq |\det A|. 
\end{equation}
\end{lemma}
\begin{lemma}{\bf (\cite{MLin1})}  \label{567}
Let  $A \in \mathbb{M}_{n}$  
 such  that  $W(A) \subset S_{\alpha}.$ Then
\begin{align}\label{f8}
 |\det A|  \leq \sec^{n}(\alpha) \det (\Re A). 
\end{align}
\end{lemma}
\begin{corollary}\label{11191c}
Let  $A,B \in \mathbb{M}_{n}$ with $W(A) , W(B)\subset S_{\alpha}$ and $0 \leq  v \leq 1$. 
\begin{itemize}
\item[(i)]
If  $0<m I_{n} \leq \Re (A^{-1}), \Re (B^{-1}) \leq MI_{n}$,
 then we have
\begin{equation}\label{20}
|\det(A \sharp_{v} B)|  
  \leq  \sec^{5n}(\alpha) K^n(h) |\det(A!_vB)|.
\end{equation}
\item[(ii)]
If  $0<m I_{n} \leq \Re (A), \Re (B) \leq MI_{n}$, 
 then we have, 
\begin{equation}\label{21}
|\det(A \sharp_{v} B)|  
  \geq  \cos^{5n}(\alpha) K^{-n}(h) |\det(A\nabla_vB)|.
\end{equation}
\end{itemize}
\end{corollary}
\begin{proof}
First,  we  prove  \eqref{20}.  Since $\det(cA)=c^n\det A$ for scalar $c>0$ and $A\in \mathbb{M}_{n}$ in general, we have 
\begin{eqnarray*}
|\det(A \sharp_{v} B)|& \leq & \sec^{n}(\alpha) \det( \Re (A \sharp_{v} B)) 
 \quad \text{(by  \eqref{f8})}  \\ 
& \leq& \sec^{5n}(\alpha) K^n(h) \det(\Re(A!_vB)
 \quad \text{(by  \eqref{f112})}  \\ 
& \leq& \sec^{5n}(\alpha) K^n(h) |\det(A!_vB)|
  \quad \text{(by  \eqref{f7})}.
\end{eqnarray*}
The  inequality \eqref{21} can be proven similarly
\begin{eqnarray*}
|\det(A \sharp_{v} B)|& \geq & \det( \Re (A \sharp_{v} B)) 
 \quad \text{(by  \eqref{f7})}   \\ 
& \geq& \cos^{4n}(\alpha) K^{-n}(h) \det(\Re(A\nabla_vB)
  \quad \text{(by  \eqref{f212})} \\ 
&\geq&  \cos^{5n}(\alpha) K^{-n}(h) |\det(A\nabla_vB )|
 \quad  \text{(by  \eqref{f8})}.
\end{eqnarray*}
This  proves  the  results  as  desired.
\end{proof}

\begin{proposition}\label{11191p}
Let  $A,B \in \mathbb{M}_{n}$ with $W(A) , W(B)\subset S_{\alpha}.$ Then  
\begin{equation*}
|\det(A \sharp B)|   \leq  \frac {\sec^{4n}(\alpha)}{2^n} |\det(I_{n}+A)|\cdot|\det(I_{n}+B)|. 
\end{equation*}

\end{proposition}
\begin{proof}
To  prove  the   assertion,  compute  
\begin{eqnarray*}
&& |\det(A \sharp B)|    \leq  \sec^{n}(\alpha) \det(\Re (A \sharp B))  \quad \text{(by  \eqref{f8})} \\
&&  \leq  \frac {\sec^{3n}(\alpha)}{2^n}
\det(\Re (A +B))   \quad \text{(by  \cite[Eq.(10)]{MLin2})}  \\
&&  \leq  \frac {\sec^{3n}(\alpha)}{2^n}
|\det(A +B)|   \quad \text{(by  \eqref{f7})}  \\
&&  \leq  \frac {\sec^{4n}(\alpha)}{2^n}
|\det(I_{n}+A)|\cdot|\det(I_{n}+B)| \quad 
\text{(by \cite[Eq.(13)]{YangCLuF})}.     
\end{eqnarray*}
\end{proof}

Note that we have the following inequality for the weighted means
$$
 |\det(A \sharp_v B)|\leq \sec^{3n}(\alpha)|\det(A\nabla_vB)|
$$
from \eqref{f8}, \eqref{f11} and \eqref{f7}.  

In the end  of  this  section,  we  give some applications  of  the  inequalities \eqref{f112}  and \eqref{f212} such  as an unitarily 
invariant norm.   A norm $\|\cdot \|_u$ is called an unitarily 
invariant norm if $\|X\|_u =\|UXV\|_u$ for any unitary matrices $U,V$ and any $X \in \mathbb{M}_{n}$. We use the symbols $v_j(X)$ and $s_j(X)$ as the $j$-th largest eigenvalue and singular value of $X$, respectively.
The following  lemmas are known.

\begin{lemma}{\bf  (Fan-Hoffman \cite[Proposition III.5.1]{BhatiaR})} \label{145}
Let  $A \in  \mathbb{M}_{n}.$  Then  
\begin{equation}\label{f1}
v_{j} (\Re A)  \leq  s_{j}(A),\quad (j=1, \cdots, n). 
\end{equation}
\end{lemma}
\begin{lemma}{\bf (\cite{DrurySLinM})} \label{146}
Let  $A \in  \mathbb{M}_{n}$ with  $W(A)\subset S_{\alpha}.$ Then  
\begin{equation}\label{f2}
s_{j}(A)   \leq \sec^{2}(\alpha) v_{j} (\Re A),\quad (j=1, \cdots, n).
\end{equation}
\end{lemma}

\begin{lemma}{\bf (\cite{Zhang})}
Let  $A \in \mathbb{M}_{n}$  with  $W(A)\subset S_{\alpha}.$ Then 
\begin{equation}\label{f99}
\| A \|_u \leq  \sec(\alpha) \|\Re(A) \|_u.
\end{equation}
\end{lemma}

\begin{corollary}\label{11c}
Let  $A,B \in  \mathbb{M}_{n}$  be  sector, that is,   $W(A),W(B) \subset S_{\alpha}$ for  some  $\alpha \in \left[0,\frac{\pi}{2}\right)$
and $0 \leq  v \leq 1$. 
\begin{itemize}
\item[(i)]
If  $0<m I_{n} \leq \Re (A^{-1}), \Re (B^{-1}) \leq MI_{n}.$
 Then,  
\begin{equation*}
 s_{j}(A  \sharp_{v} B)   \leq  
 \sec^{6}(\alpha) K(h) s_{j}(A!_vB),
\end{equation*}
\item[(ii)]
If  $0<m I_{n} \leq \Re (A), \Re (B) \leq MI_{n}.$
 Then,  
\begin{equation*}
 \cos^{6}(\alpha)  K^{-1}(h) 
s_{j}( A\nabla_vB)  \leq 
   s_{j}(A  \sharp_{v} B).
\end{equation*}
\end{itemize}
\end{corollary}
\begin{proof}
A  simple  computation  shows  that 
\begin{eqnarray*}
 s_{j}(A  \sharp_{v} B) &  \leq  &
 \sec^{2}(\alpha)  s_{j}( \Re (A\sharp_{v} B)) \quad \text{(by  \eqref{f2})}   \\
&\leq &
 \sec^{6}(\alpha) K(h) s_{j}(\Re (A!_vB) \quad \text{(by  \eqref{f112})}  \\
&\leq &
 \sec^{6}(\alpha)  K(h)  s_{j}(A!_vB)
  \quad \text{(by  \eqref{f1})}.   
\end{eqnarray*}
It  is  easy  to  observe  that 
\begin{eqnarray*} 
s_{j}(A \sharp_{v} B) 
 &\geq &  s_{j}(\Re(A \sharp_{v} B) )  \quad  \text{(by  \eqref{f1})}\\
&\geq &\cos^{4}(\alpha)  K^{-1}(h) 
s_{j}( \Re (A\nabla_vB) 
 \quad \text{(by  \eqref{f212})} \\ 
&\geq &
\cos^{6}(\alpha)  K^{-1}(h) s_{j}(A\nabla_vB) \quad
 \text{(by  \eqref{f2})}. \label{2236}
\end{eqnarray*}
\end{proof}

\begin{remark}\label{11r}
In  special  case such that $\alpha=\frac{\pi}{4}$, 
we  have  the  following  inequalities  for  accretive-disipative  matrices $A,B \in  \mathbb{M}_{n}$  and $0 \leq  v \leq 1$.
\begin{itemize}
\item[(i)]
If  $0<m I_{n} \leq \Re (A^{-1}), \Re (B^{-1}) \leq MI_{n}.$
 Then,  
\begin{equation*}
 s_{j}(A  \sharp_{v} B)   \leq  8K(h) s_{j}(A!_vB).
\end{equation*}
\item[(ii)]
If  $0<m I_{n} \leq \Re (A), \Re (B) \leq MI_{n}.$
 Then  
\begin{equation*}
 \frac{1}{8} K^{-1}(h)  
s_{j}( A\nabla_vB)  \leq    s_{j}(A  \sharp_{v} B). 
\end{equation*}
\end{itemize}
\end{remark}

\begin{corollary}\label{1c}
Let  $A,B \in \mathbb{M}_{n}$ with  $W(A) , W(B)\subset S_{\alpha}.$ Then  
for  any  unitarily  invariant  norm  $\|\cdot \|_u$  on  $\mathbb M_n$, we have the following inequalities.
\begin{itemize}
\item[(i)]
If  $0<m I_{n} \leq \Re (A^{-1}), \Re (B^{-1}) \leq MI_{n}$, then we have  
\begin{equation*}
\|A\sharp_{v} B \|_u  \leq 
 \sec^{5}(\alpha)  K(h)\| A!_vB \|_u.
\end{equation*}
\item[(ii)]
If  $0<m I_{n} \leq \Re (A), \Re (B) \leq MI_{n}$, then we have
\begin{equation*}
\|A\sharp_{v} B \|_u \ge \cos^{5}(\alpha) K^{-1}(h)   \|A\nabla_vB\|_u   
\end{equation*}
\end{itemize}
\end{corollary}
\begin{proof}
 We  can  show  that  the following chain of inequalities  for  a unitarily  invariant  norm:
 \begin{eqnarray*}
&& \| A  \sharp_{v} B  \|
 \leq  \sec(\alpha) \|\Re(A  \sharp_{v} B) \|  \quad  {\text {(by  \eqref{f99})}}  \\
&& \hspace*{15mm} \leq  \sec^{5}(\alpha) K(h)\|\Re (A!_vB) \|   \quad {\text {(by  \eqref{f112})}}   \\
&& \hspace*{15mm}  \leq  \sec^{5}(\alpha) K(h)\|A!_vB\|.  
\end{eqnarray*}
This  proves  the  first  inequality.  
The  second  inequality can be proven similarly
\begin{eqnarray*}
&& \| A  \sharp_{v} B  \|_u \geq  \|\Re(A  \sharp_{v} B) \|_u    \geq     \cos^{4}(\alpha) K^{-1}(h)\|\Re(A\nabla_vB) \|_u \quad  {\text {(by  \eqref{f212})}} \\
&& \hspace*{17.5mm}  \geq  \cos^{5}(\alpha) K^{-1}(h)   
\|A\nabla_vB\|_u.  \quad
{\text {(by  \eqref{f99})}} 
\end{eqnarray*}
\end{proof}

\begin{remark}\label{1r}
In  special  case such that  $\alpha=\frac{\pi}{4}$, 
we  have  the  following  inequalities  for  accretive-disipative  matrices $A,B \in \mathbb{M}_{n}$ and  any  unitarily  invariant  norm  $\|\cdot \|_u$  on  $\mathbb M_n$,
\begin{equation*}
4\sqrt{2} K^{-1}(h)   \|A\nabla_vB\| _u  \leq 
  \|A\sharp_{v} B \|_u  \leq 
\frac {1}{4\sqrt{2}}  K(h)\| A!_vB \|_u.
\end{equation*}
\end{remark}

\begin{proposition}\label{1p}
Let  $A,B \in \mathbb{M}_{n}$  such  that  $W(A) , W(B)\subset S_{\alpha}.$ Then  
\begin{equation*}
\|A\sharp  B  \|_u  \leq 
\frac {\sec^{5}(\alpha)}{2} \| I_{n}+A  \|_u\cdot \|I_{n}+B \|_u.
\end{equation*}
\end{proposition}
\begin{proof}
\begin{eqnarray*}
\|  A  \sharp B  \|_u 
&  \leq & \frac {\sec^{3}(\alpha)}{2} \| A +B \|_u \quad  {\text {(by \cite[Eq.(14)]{MLin2})}}  \\
& \leq & \frac {\sec^{5}(\alpha)}{2} \| I_{n}+A  \|_u \cdot\|I_{n}+B \|_u \quad {\text {(by   \cite[Corollary 2.8]{YangCLuF})}}.  
\end{eqnarray*}
\end{proof}

\vskip 0.3 true cm

{\tiny (L. Nasiri) Department of Mathematics and computer science, Faculty of science, Lorestan
	University, Khorramabad, Iran}
{\tiny \textit{E-mail address:} leilanasiri468@gmail.com}\\

{\tiny(S. Furuichi)
 Department of Information Science, College of Humanities and Sciences, Nihon University, 3-25-40, Sakurajyousui, Setagaya-ku, Tokyo, 156-8550, Japan}
 {\tiny   \textit{E-mail address:} furuichi@chs.nihon-u.ac.jp}

\end{document}